\documentclass{amsart}

\usepackage{amssymb, amsmath, latexsym, a4}
\usepackage[latin1]{inputenc}
\usepackage[all]{xy}

\newtheorem{De}{Definition}
\newtheorem{Th}[De]{Theorem}
\newtheorem{Pro}[De]{Proposition}
\newtheorem{Le}[De]{Lemma}
\newtheorem{Co}[De]{Corollary}

\def\xto#1{\xrightarrow[]{#1}}

\begin{document}

\title{Endomorphisms in short exact sequences}

\author[M. Pirashvili]{Mariam  Pirashvili}
\thanks{Research was partially supported by the grant ''DI/27/5-103/12,D-13/2 Homological and categorical methods in topology, algebra and theory of stacks''}

\maketitle

\section{Introduction}
Recently,  Passi, Singh and Yadav constructed interesting exact sequences of automorphisms involved in group extensions with an abelian kernel, published in \cite{passi}. We wish to shed new light on these exact sequences. First, we show that the exactness holds even for endomorphisms. This, it turns out, is a lot easier to derive from known results, such as the five-term exact sequence. Moreover, these endomorphisms carry exotic ring structures and one obtains the exact sequence of automorphisms by restricting our exact sequence to quasi-regular elements, giving an alternative way to prove the exact sequence due to Passi, Singh and Yadav.

The second section of the paper states the main results involving two main theorems which are then expanded and proved in sections 3 and 4 respectively. In the last section, the ring structures mentioned above are explicitly calculated for a couple of examples.

\section{The main results}\label{1}
Let $0\to N\xrightarrow{i} G\xrightarrow{p} Q \to 1$ be an abelian extension of groups. We use additive notation for $N$ and  $G$, though we do not assume that $G$ is abelian.

It is generally known that $N$ has a $Q$-module structure induced by the conjugation in $G$. We let $End_Q(N)$ be the set of all $Q$-module endomorphisms of $N$, which is a ring under pointwise addition and composition.

We denote by $End^{N,Q}(G)$ the set of all endomorphisms of $G$ which centralise $N$ and induce the identity on $Q$,
 i.e. the set of all homomorphisms $\alpha:G\to G$ which fit in the following commutative diagram:
$$\xymatrix{0\ar[r]&N\ar@{=}[d]\ar[r]&G\ar[d]^{\alpha}\ar[r] &Q\ar@{=}[d]\ar[r] &1\\
0\ar[r]& N\ar[r] &G\ar[r] &Q\ar[r] &1.}$$

We denote by $End^Q_N(G)$ the set of all endomorphisms of $G$ which normalise $N$ and induce the identity on $Q$, i.e. the set of all homomorphisms $\alpha:G\to G$ such that $\alpha(N)\subset N$. Denote by $\beta$ the restriction of $\alpha$ on $N$. Then one has a commutative diagram
$$\xymatrix{0\ar[r]&N\ar[d]_{\beta}\ar[r]&G\ar[d]^{\alpha}\ar[r]^p&Q\ar@ {=}[d]\ar[r] &1\\
0\ar[r]& N\ar[r] &G\ar[r]^p&Q\ar[r] &1.}$$ It is clear that $End^{N,Q}(G)\subset End^Q_N(G)$. We also have a map
$$\rho:End^Q_N(G)\to End_Q(N)$$ given by $\rho(\alpha)=\beta$. 
The composition of endomorphisms of $G$ yields a multiplicative monoid structure on $End^Q_N(G)$. It turns out that there are other binary operations on $End^Q_N(G)$ with nicer properties.
\begin{Th} \label{th1} \begin{enumerate} 
\item There is an associative, but not necessarily unital ring structure on $End^Q_N(G)$, such that the ''addition'' $\boxplus$ and ''multiplication'' $ \boxtimes$ are defined by
$$(\alpha_1\boxplus \alpha_2)(x):=\alpha_1(x)-x+\alpha_2(x).$$
$$(\alpha_2\boxtimes \alpha_1)(x):=\alpha_2\circ \alpha_1(x)-\alpha_1(x)+x-\alpha_2(x)+x,$$
Here $\alpha_1,\alpha_2\in End^Q_N(G)$ and $x\in G$. Moreover the ''zero'' element of this ring is $id_G\in End^Q_N(G)$.
\item $End^{N,Q}(G)$ is a square zero  ideal  of $End^Q_N(G)$ and the additive structure on $End^{N,Q}(G)$ induced from $End^Q_N(G)$ is given by the composition of endomorphisms.
\item The map $\rho:End^Q_N(G)\to End_Q(N)$ is a homomorphism of rings and it fits in the following exact sequence of abelian groups
\begin{equation}\label{eta}0\to End^{N,Q}(G)\to End^Q_N(G)\xto{\rho} End_Q(N)\xto{\eta} H^2(Q,N)\xto{p^*} H^2(G,N).\end{equation}
 \end{enumerate}
\end{Th}

In the definition of $End^{N,Q}(G)$ and $End^Q_N(G)$ instead of endomorphisms we could consider automorphisms $\alpha:G\to G$. In this way one obtains the groups $Aut^{N,Q}(G)$ and $Aut^Q_N(G)$. By the 5-Lemma we have $$End^{N,Q}(G)=Aut^{N,Q}(G).$$
By definition $End^Q_N(G)$ contains $Aut^Q_N(G)$ and the homomorphism $\rho$ restricts to a homomorphism of groups $\rho':Aut^Q_N(G)\to Aut_Q(N)$. Here $Aut_Q(N)$ consists of all $Q$-automorphisms $N\to N$.
\begin{Co}\label{co1} The exact sequence in Theorem \ref{th1} (3) restricted to quasi-regular elements yields an exact sequence 
$$1\to Aut^{N,Q}(G)\xto{i'} Aut^Q_N(G)\xto{\rho'} Aut_Q(N)\xto{\eta} H^2(Q,N),$$
where all maps are group homomorphisms except $\eta$, which is a pointed map.
\end{Co}
We denote by $End^N(G)$ (resp. $Aut^N(G)$) the set of all endomorphisms (resp. automorphisms) of $G$ which centralise $N$, i.e. the set of all homomorphisms (resp. isomorphisms) $\alpha:G\to G$ such that the diagram commutes:
$$\xymatrix{0\ar[r]&N\ar@{=}[d]\ar[r]&G\ar[d]^{\alpha}\ar[r]&Q\ar[d]^{\bar{\alpha}}\ar[r] &1\\
0\ar[r]& N\ar[r] &G\ar[r] &Q\ar[r] &1.}$$
We denote by $End^N(Q)$ the set of endomorphisms of $Q$ which preserve the action of $Q$ on $N$, i.e. the set of endomorphisms $\phi$ for which $x\cdot n=\phi(x)\cdot n$ for all $x\in Q$ and $n\in N$.
\begin{Th}\label{th2} One has an exact sequence of pointed sets
\begin{equation}0\to End^{N,Q}(G)\xto{\bar{i}} End^N(G)\xto{\bar{\rho}}End^N(Q)\to H^2(Q,N),\end{equation}
where $\bar{i}$ and $\bar{\rho}$ are monoid homomorphisms.
\end{Th}
If one restricts the exact sequence of monoids in Theorem \ref{th2} on invertible elements we obtain the following result.
\begin{Co}\label{co2} 
The sequence
$$1\to Aut^{N,Q}(G)\to Aut^N(G)\to Aut^N(Q)\xto{\bar{\eta}} H^2(Q,N)$$
is exact, where all maps are group homomorphisms except $\bar{\eta}$, which is  a pointed map.
\end{Co}
The exact sequences in Corollary \ref{co1} and Corollary \ref{co2} can be found in Theorem 1 of the paper \cite{passi} by Passi, Singh and Yadav, where  $Aut_Q(N)$ is denoted by $C_1$ and $Aut^N(Q)$ by $C_2$. The proofs given in  \cite{passi} use completely different methods. Our proofs are based on the well-known 5-term exact sequence in group cohomology \cite{ht} and the low dimensional exact sequence of non-abelian cohomology associated to a central extension \cite{serre}.

Finally on notations. If $G$ is a group and $U$ is a $G$-group, we denote by $Z^1(G,U)$ the set of all crossed homomorphisms $G\to U$. This is an abelian group, provided $U$  is abelian.  Sometimes crossed homomorphisms are called $1$-cocycles.

\section{Proofs of Theorem \ref{th1} and Corollary \ref{co1}}

\subsection{Ring structure on $Z^1(G,N)$} A ring in this paper means an associative  but not necessarily unital ring.  Recall that we have actions of $G$ and $Q$ on $N$ induced by the conjugation action of $G$. In particular, $N$ acts on itself trivially.
Let us take a closer look at $Z^1(G,N)$. We equip the abelian group $Z^1(G,N)$ with a ring structure.  

\begin{Le} \label{Res}
Let $\phi\in Z^1(G,N)$. Restricted to $N$, the crossed homomorphism $\phi$ is a $Q$-homomorphism.
\end{Le}

\begin{proof}
Let $m,n\in N$. Then we have
$$\phi\circ i(m+n)=\phi(m)+m\cdot\phi(n)=\phi(m)+\phi(n),$$
as action of $N$ on itself is trivial. Next, let $\bar y=p(y)\in Q$. Then for $n\in N$
\begin{align*}
\phi\circ i(\bar y\cdot n)&=\phi\circ i(y+n-y)=\phi(y)+y+\phi(n-y)-y\\
&=\phi(y)+y+\phi(n)+n\cdot\phi(-y)-y=\phi(y)+y+\phi(n)+\phi(-y)-y\\
&=\phi(y)+y+\phi(-y)+\phi(n)-y=\phi(y)+y+\phi(-y)-y+y+\phi(n)-y\\
&=\phi(y-y)+y+\phi(n)-y\\
&=\bar y\cdot(\phi\circ i(n))
\end{align*}
\end{proof}

\begin{Le}
For $\phi,\psi\in Z^1(G,N)$ we define
$$\phi\diamond\psi(x):=\phi\circ i\circ\psi(x).$$
Then $\phi\diamond\psi\in Z^1(G,N)$.
\end{Le}
\begin{proof} To simplify notations, we omit $i$. For $x,y\in G$ we have
\begin{align*}
\phi\diamond\psi(x+y)&=\phi\circ i\circ\psi(x+y)=\phi(\psi(x+y))\\
&=\phi(\psi(x)+x\cdot\psi(y))\\
&=\phi(\psi(x))+\phi(x\cdot \psi(y))\\
&=\phi(\psi(x))+x\cdot \phi(\psi(y)),
\end{align*}
by Lemma \ref{Res}. So $\phi\diamond\psi\in Z^1(G,N)$. \end{proof}
\begin{Le}\label{diamondring} The binary operation $\diamond$ yields a ring structure on $Z^1(G,N)$ and the induced map
$$Res:Z^1(G,N)\to End_Q(N)$$
is a ring homomorphism.
\end{Le}
\begin{proof}
All ring axioms are easy to check. The only one that is not entirely straightforward is the right-sided distributivity axiom. Let us check this. Let $\phi,\xi,\psi\in Z^1(G,N)$. Then we have
\begin{align*}
\phi\diamond(\xi+\psi)(x)&=\phi\circ i\circ(\xi+\psi)(x)=\phi\circ i(\xi(x)+\psi(x))\\
&=\phi\circ( i(\xi(x))+i(\psi(x))), \ {\rm as} \  i\ {\rm is \ an \ homomorphism,}\\
&=\phi( i(\xi(x)))+i(\xi(x))\cdot\phi(i(\psi(x)))\\
&=\phi( i(\xi(x)))+\phi(i(\psi(x)))=\phi\diamond\xi(x)+\phi\diamond\psi(x),
\end{align*}
as $\xi(x)\in N$, and the action of $N$ on itself is trivial.
\end{proof}
\subsection{Endomorphisms and crossed homomorphisms}
The following well-known and straightforward lemma is a main tool in translating problems relating to endomorphisms and automorphisms of group extensions to the framework of group cohomology. Thanks to this, we can use the results of group cohomology to study endomorphisms and automorphisms of group extensions. For convenience's sake, we present the proof here.
\begin{Le}\label{iso}
Let $G$ be a group and $\alpha:G\to G$ be a map such that $\alpha(x)=\phi(x)+ x$ for $\phi:G\to G$. Then $\alpha$ is a group homomorphism if and only if $\phi \in Z^1(G,G)$, where $G$ acts on itself via conjugation. In particular there is a bijection $End(G)\cong Z^1(G,G)$.
\end{Le}

\begin{proof}
Let $x,y\in G$. We have $\alpha$ is a homomorphism if and only if
$$\phi(x+y)+x+y=\alpha(x+y)=\alpha(x)+\alpha(y)=\phi(x)+x+\phi(y)+y.$$
Therefore, we have
$$\phi(x+y)=\phi(x)+x+\phi(y)-x.$$
As this is the cocycle condition for $\phi$, we see that $\phi \in Z^1(G,G)$.
\end{proof}

Recall that $End^{N,Q}(G)=Aut^{N,Q}(G)$ is a group with respect to composition.
\begin{Le}\label{Q,N}
We have an isomorphism of groups between $End^{N,Q}(G)=Aut^{N,Q}(G)$ and $Z^1(Q,N)$.
\end{Le}

This is a well-known fact, and a proof can be found on p. 192 of \cite{wells} or p. 66 of \cite{rob}.

We can use the fact that $Z^1(G,N)$ is abelian to deduce:
\begin{Co}
$Aut^{N,Q}(G)$ is a commutative group under composition.
\end{Co}
\begin{Le}\label{G,N}
There is a bijection between $End^Q_N(G)$ and $Z^1(G,N)$. Therefore we have ring structure on $End^Q_N(G)$.
\end{Le}
\begin{proof}
Let $\alpha \in End^Q_N(G)$. We can write $\alpha(x)=\psi(x)+x$ for $x\in G$ and some map $\psi$ which we will now show is an element of $Z^1(G,N)$, thus defining the required bijection. By Lemma \ref{iso}  $\psi$ is a crossed homomorphism. Next, we check that $\psi(x)\in N$. Let us denote $p(x)$ by $\bar x$. Since $p\alpha =p$,  we have $p(\alpha(x))=p(\psi(x)+x)=p(x)$. From this it follows that $\bar \psi(x)=0$, and therefore $\psi(x)\in N$.
\end{proof}
\subsection{Transporting the ring structure}\label{transport} 
Now let us look at what happens to the two operations in $Z^1(G,N)$ under the bijection established in Lemma \ref{G,N}. First, we take a look at pointwise addition. For $\alpha_1, \alpha_2\in End^{N,Q}(G)$ and $\psi_1,\psi_2,\psi\in Z^1(Q,N)$ such that $\alpha_1(x)=\psi_1(x)+x$, $\alpha_2(x)=\psi_2(x)+x$ and $\psi(x)=\psi_1(x)+\psi_2(x)$, we want to define an operation
$$\boxplus:End^Q_N(G)\times End^Q_N(G)\to End^Q_N(G)$$
such that there exists some $\alpha\in End^Q_N(G)$ with
$$\alpha(x)=(\alpha_1\boxplus \alpha_2)(x)=\psi(x)+x.$$
Let us take a look at $\psi(x)$.
$$\psi(x)=\psi_1(x)+\psi_2(x)=\alpha_1(x)-x+\alpha_2(x)-x.$$
So we define
$$(\alpha_1\boxplus \alpha_2)(x):=\alpha_1(x)-x+\alpha_2(x)-x+x=\alpha_1(x)-x+\alpha_2(x).$$
For this operation, it is easy to check that the neutral element is $\alpha(x)=x$, and the inverse element for $\alpha$ is $-\alpha(x):=x-\alpha(x)+x.$

In the same manner we wish to define a second operation
$$\boxtimes :End^Q_N(G)\times End^Q_N(G)\to End^Q_N(G)$$
such that for $\alpha(x)=(\alpha_2\boxtimes \alpha_1)(x)$ and $\psi( x)= \psi_2\diamond \psi_1(x)$ we have $\alpha(x)=\psi(x)+x$.
Let us look at $\alpha_2\circ \alpha_1(x)$. Composition is defined in $End^Q_N(G)$, though we do not have inverses in general, giving $End^Q_N(G)$ a monoidal structure.
\begin{align*}
\alpha_2\circ \alpha_1(x)&= \alpha_2(\psi_1(x)+x)=\psi_2(\psi_1(x)+x)+\psi_1(x)+x\\
&=\psi_2(\psi_1(x))+\psi_1(x)+\psi_2(x)-\psi_1(x)+\psi_1(x)+x\\
&=\psi_2\diamond \psi_1(x)+\psi_2(x)+\psi_1(x)+x.
\end{align*}
From this we can work backwards to get
\begin{align*}
\alpha_2\circ \alpha_1(x)&=\psi_2\diamond \psi_1(x)+\psi_2(x)+\alpha_1(x)\\
\alpha_2\circ \alpha_1(x)-\alpha_1(x)&=\psi_2\diamond \psi_1(x)+\psi_2(x)+x-x\\
\alpha_2\circ \alpha_1(x)-\alpha_1(x)+x&=\psi_2\diamond \psi_1(x)+\alpha_2(x)\\
\alpha_2\circ \alpha_1(x)-\alpha_1(x)+x-\alpha_2(x)&=\psi_2\diamond \psi_1(x).\\
\end{align*}
So we have
$$\alpha(x)=(\alpha_2\boxtimes \alpha_1)(x)=\alpha_2\circ \alpha_1(x)-\alpha_1(x)+x-\alpha_2(x)+x=\psi_2\diamond \psi_1(x)+x.$$

\subsection{Proof of Theorem \ref{th1}}
Part (1) of Theorem \ref{th1} follows from the calculations in Section \ref{transport} and Lemma \ref{diamondring}.

Part (2) is trivial.  To show part (3) let us observe that we have the following exact sequence:
$$0\to Z^1(Q,N)\xrightarrow{Inf} Z^1(G,N)\xrightarrow{Res} End_{Q}(N)\to H^2(Q,N)\to H^2(G,N)$$
thanks to \cite[Exact sequence 8.1 on p. 202]{ht}). Now the result follows from  Lemmas \ref{Q,N} and \ref{G,N}.

\subsection{Operation $\ast$ and the quasi-regular elements in rings}
Let $R$ be a ring. We define an operation $\ast:R\times R\to R$ by
$$r\ast s:=r+s+rs.$$
Following \cite{kap} we define an element $r\in R$ to be \emph{quasi-regular} if there exists some element $s\in R$ with $r\ast s=0=s\ast r$. Let $QR(R)$ denote the set of all such elements of $R$. Then $QR(R)$ is a group under $\ast$ with neutral element $0$. For $r\in QR(R)$, the inverse element is such an $s$ as given in the definition of a quasi-regular element. Moreover, if $R$ is a unital ring and $U(R)$ denotes the set of invertible elements of $R$, we have a group isomorphism $\phi:QR(R)\to U(R)$ given by $\phi(r)=1+r$ (see  \cite{kap}).

An ideal $I$ of a ring $R$ is a \emph{square zero ideal} if for all $a_1,a_2\in I$ one has $a_1a_2=0$.

\begin{Pro}\label{QR}
Let $0\rightarrow I\xrightarrow{i}R\xrightarrow{p}S\rightarrow 0$ be a short exact sequence such that $R$ and $S$ are rings, $p$ is a ring homomorphism and $I$ is the square zero ideal. Then
\begin{enumerate}
\item for $r\in R$ with $p(r)\in QR(S)$, we have $r\in QR(R)$ and
\item the sequence $0\rightarrow I\xrightarrow{i_2}QR(R)\xrightarrow{p_2}QR(S)\rightarrow 0$ is also exact.
\end{enumerate}
\end{Pro}
\begin{proof}
Let $p(r)=s\in QR(S)$. Then there exists an $\bar s\in QR(s)$ with $s\ast\bar s=0$. As $p$ is surjective, we must have an $\bar r\in R$ such that $p(\bar r)=\bar s$. We also have that $p(r\ast\bar r)=s\ast\bar s$ and therefore $r\ast\bar r\in I$. Suppose $r\ast\bar r=q\in I$. Then let $z = \bar r\ast (-q)$. We have $p(z)=p(\bar r)=\bar s$ and
$$r\ast z=r\ast (\bar r\ast(-q))=(r\ast \bar r)\ast (-q)=q\ast(-q)=0,$$
as $q\in I$, showing that $r\in QR(R)$, as we wanted.

The second part is obvious. The surjectivity of $p_2$, which is perhaps least obvious, follows from the first part.
\end{proof}

\begin{Le}
We have $QR(End_Q(N))\simeq Aut_Q(N)$.
\end{Le}
\begin{proof}
Trivial, as for a unital ring, there is a bijection between the quasi-regular and invertible elements.
\end{proof}
\begin{Le} The  operation $\ast$ for the ring $End_N^Q(G)$ is the same as $\circ$. In particular
we have 
$QR(End_N^Q(G))\simeq Aut_N^Q(G)$.
\end{Le}
\begin{proof}  Take $f,g\in End_N^Q(G)$. Then by definition of $f\boxtimes g$ we have  
\begin{align*}  f\circ g&=f\boxtimes g -id+f-id+g\\
&=f\boxtimes g -id +f\boxplus g\\
&=(f\boxtimes g)\boxplus(f\boxplus g)\\
&=f\ast g.
\end{align*}
Since $id$ is the zero element in the ring $End_N^Q(G)$, we see that $f$ is regular if and only if there exists a $g$ such that $f\circ g=id$. 
\end{proof}
\subsection{Proof of Corollary \ref{co1}}
The exact sequence in Theorem \ref{th1} gives rise to an extension of rings
$$0\to End^{N,Q}(G)\to End^Q_N(G)\to Im(\rho)\to 0$$
and a monomorphism of rings $Im(\rho)\to End_Q(N)$. We can, therefore, use Proposition \ref{QR} to get that
\begin{equation}\label{aut}
0\to Aut^{N,Q}(G)\to Aut^Q_N(G)\to QR(Im(Res))\to 0
\end{equation}
is exact as well. We also have a monomorphism of groups $QR(Im(Res))\to QR(End_Q(N))=Aut_Q(N)$. Hence it is clear that in the commutative diagram below
$$\xymatrix{0\ar[r]&End^{N,Q}(G)\ar@{=}[d]\ar[r]&End^Q_N(G)\ar[r]^{Res}&End_Q(N)\ar[r]^{\phi}&H^2(Q,N)\ar@{=}[d]\ar[r] &\\
0\ar[r]&Aut^{N,Q}(G)\ar[r]&Aut^Q_N(G)\ar@{^{(}->}[u]\ar[r]^{R}&Aut_Q(N)\ar@{^{(}->}[u]_{p}\ar[r]^{\psi}&H^2(Q,N)}$$
the bottom row is exact except possibly at the place $Aut_Q(N)$. We want to show that $Im \ R=Ker \ \psi$. First, let $f\in Im \ R$. We have that $p(f)\in Im \ Res$ and $\phi(p(f))=0$. But then as the diagram commutes, we must have $f\in Ker \ \psi$.

Now, let $f\in Ker \ \psi$. Then $p(f)\in Ker \ \phi$ and so there exists a $g\in End^Q_N(G)$ such that $Res(g)=p(f).$ So $g$ maps to an element of $QR(End_Q(N))$. But then according to the first part of Proposition \ref{QR}, this implies that $g\in QR(End^Q_N(G))=Aut^Q_N(G)$ and so $f\in Im \ R$.

\section{Proofs of Theorem \ref{th2} and Corollary \ref{co2}}
\subsection{Endomorphisms and crossed homomorphism again}
\begin{Le}\label{15}
There is a bijection between $End^N(G)$ and $Z^1(Q,C_G(N))$, where $C_G(N)$ denotes the centraliser of $N$ in $G$.
\end{Le}
\begin{proof}
Let $\alpha \in End^N(G)$. Then we can write $\alpha(x)=\psi(x)+x$ for $x\in G$ and some map $\psi$ which we will now show is an element of $Z^1(Q,C_G(N))$, thus defining the required bijection.

First, we need to check that $C_G(N)$ is a $Q$-group, and therefore $Z^1(Q,C_G(N))$ is defined. For an $x\in G$, let us denote $p(x)$ by $\bar x$. Then for a $g\in C_G(N)$, $\bar x \cdot g=x+g-x$. This action is well-defined. It is easy to see that $\bar x\cdot g \in C_G(N)$ and for a $y\in G$ with $p(y)=\bar x$ we have $y+g-y=x+n+g-n-x=x+n-x$.

The cocycle condition follows from Lemma \ref{iso} and the fact that $\psi$ is actually defined can be proved the same way as in Lemma \ref{Q,N}. Next, we check that $\psi(\bar x)\in C_G(N)$. Let $g\in Im \ \psi$. So there is some $\bar x\in Q$ with $\psi(\bar x)=g$. Then for every $n\in N$, we have
\begin{align*} n+g-n&=n+\alpha(x)-x-n\\
&=\alpha(n+x)-(n+x) \ {\rm as} \  \alpha(n)=n,\\
&=\psi(\overline{n+x})=\psi(\bar x).
\end{align*}
So $n+g-n=g$ and therefore $g\in C_G(N)$.
\end{proof}

We denote by $\bar Q$ the quotient $C_G(N)/N$ which is part of the central extension
$$0\to N\to C_G(N)\to \bar Q\to 1.$$
We will now show that $\bar Q$ is the set of all elements of $Q$ whose action on $N$ is trivial. 
The following diagram
$$\xymatrix{0\ar[r]&N\ar@{=}[d]\ar[r]&G\ar[r]&Q\ar[r]&1\\
0\ar[r]&N\ar[r] &C_G(N)\ar@{^{(}->}[u]_{j}\ar[r] &{\bar Q}\ar[u]_{\bar j}\ar[r] &1}$$
shows that $\bar{j}$ is a monomorphism. So $\bar Q\subset Q$. 
Next, we show that all $x\in Q$ with $x\cdot n=n$ for every $n\in N$ are elements of $\bar Q$. Let such an $x$ be given. Then for a $g\in G$ with $p(g)=x$, we have $x\cdot n=g+n-g=n$. So $g+n=n+g$, and therefore $g\in C_G(N)$. Therefore $x\in \bar Q$. 

Thus $\bar{Q}=Ker(Q\to Aut(N))$ and hence the conjugation yields an action of $Q$ on $\bar Q$. So $Z^1(Q,\bar Q)$ is defined and we have the following:
\begin{Le}\label{16}
There is a bijection between $End^N(Q)$ and $Z^1(Q,\bar Q)$.
\end{Le}
\begin{proof}
Let $\alpha\in End^N(Q)$ with $\alpha(x)=\phi(x)+x$ for $x\in Q$. We want to show that $\phi\in Z^1(Q,\bar Q)$. The cocycle condition is again given by Lemma \ref{iso}. What is left to show is that $\phi(x)\in \bar Q$, i.e. that the action of $\phi(x)$ on any $n\in N$ is trivial. For $\alpha$ we have $x\cdot n=\alpha(x)\cdot n=(\phi(x)+x)\cdot n=\phi(x)\cdot(x\cdot n)=x\cdot n$. So $\phi\in Z^1(Q, \bar Q)$.
\end{proof}
\begin{Le}\label{cex}
 Let $G$ be a group and let
$$0\to A\to B\to C\to 0$$
be an central extensions of $G$-groups. Then one has an exact sequence of pointed sets
$$0\to Z^1(G,A)\to Z^1(G,B)\to Z^1(G,C)\to H^2(G,A).$$
\end{Le} 
\begin{proof} This is a variant of the  exact sequence \cite[Proposition 43, p. 55]{serre} and has the same proof.
\end{proof}

\subsection{Proof of Theorem \ref{th2}} Apply Lemma  \ref{cex}  to the central extenion $0\to N\to C_G(N)\to \bar Q\to 1$ to get the exact sequence 
$$0\to Z^1(Q,N)\to Z^1(Q,C_G(N))\to Z^1(Q,\bar Q)\to H^2(Q,N).$$
The rest follows from Lemma \ref{15} and  Lemma \ref{16}.

\subsection{Proof of Corollary \ref{co2}}
Let us look at the following commutative diagram:
$$\xymatrix{0\ar[r]&End^{N,Q}(G)\ar@{=}[d]\ar[r]&End^N(G)\ar[r]^{q}&End^N(Q)\ar[r]^{\phi}&H^2(Q,N)\ar@{=}[d]\\
0\ar[r]&Aut^{N,Q}(G)\ar[r]&Aut^N(G)\ar@{^{(}->}[u]\ar[r]^{R}&Aut^N(Q)\ar@{^{(}->}[u]_{p}\ar[r]^{\psi}&H^2(Q,N).}$$
Exactness at $Aut^{N,Q}(G)$ and $Aut^N(G)$ are obvious. Take an element $\bar{\alpha}\in Ker(\psi)$. Then there exists an element $\alpha \in End^N(G)$ such that $p(\bar{\alpha})=q(\alpha).$ By the definition of the map $p$ one has the diagram 
$$\xymatrix{0\ar[r]&N\ar@{=}[d]\ar[r]&G\ar[d]^{\alpha}\ar[r]&Q\ar[d]^{\bar{\alpha}}\ar[r] &1\\
0\ar[r]& N\ar[r] &G\ar[r] &Q\ar[r] &1.}$$
By the assumtion $\bar{\alpha}$ is an automorphism. By the 5-Lemma $\alpha$ is also an automorphism. Thus $\alpha \in Aut^N(G)$ and $R(\alpha)=\bar{\alpha}$ and exactness follows.

\section{Examples}\label{Ex}
Given a not necessarily unital ring $R$ and a not necessarily unital $R$-$R$-bimodule $S$, we recall the definition of the semidirect product of rings $S\rtimes R$. As a set, this is the cartesian product of $S$ and $R$. Addition is componentwise, and multiplication is given by
$$(s_1,r_1)* (s_2,r_2)=(r_1s_2+s_1r_2,r_1r_2).$$

Given a split extension $0\to N\to N\rtimes Q\to Q\to 0$, where $N\rtimes Q$ is the semidirect product of groups, one can see that the map $\rho$ in the exact sequence (\ref{eta}) is surjective and hence the connecting map $\eta$ is the zero map. In the short exact sequence of rings $0\to End^{N,Q}(G)\to End^Q_N(G)\xto{\rho} End_Q(N)\to 0$, the ring $End^Q_N(G)$ is a semidirect product of $End^{N,Q}(G)$ and $End_Q(N)$.

The first example is a special case of the above. The second example does not begin with a split extension, but ultimately leads to an exact sequence of rings, where the second term is again a semidirect product of rings.

\subsection{Example 1}\label{Dn}

Let $N=C_n$ and $Q = C_2$. In the abelian extension of groups $0\to C_n\to D_n\to C_2\to 0$, the group $D_n$ is the semidirect product of $C_n$ and $C_2$.

We have $End_{C_2}(C_n)=\mathbb{Z}/ n\mathbb{Z}$, $End^{C_2,C_n}(D_n)=\mathbb{Z}/ n\mathbb{Z}$ and $End_{C_n}^{C_2}(D_n)=\mathbb{Z}/ n\mathbb{Z}\rtimes \mathbb{Z}/ n\mathbb{Z}$ as defined above. The left action of $End_{C_2}(C_n)$ on $End^{C_2,C_n}(D_n)$  is ring multiplication, while the right action is zero.

In terms of endomorphisms, the ring structure on $End_{C_n}^{C_2}(D_n)$ is as follows:
Let $f_{k,l}\in End_{C_n}^{C_2}(D_n)$ and let $x,y$ be the generators of $D_n$. Then $f_{k,l}(x)=xy^k$ and $f_{k,l}(y)=y^{l+1}$.

We have $f_{k,l}\boxplus f_{p,q}=f_{k+p,l+q}$ and $f_{k,l}\boxtimes f_{p,q}=f_{lp,lq}$.

Restricted on $End^{C_2,C_n}(D_n)$, we have $f_{k,0}\boxplus f_{p,0}=f_{k+p,0}$ and $f_{k,0}\boxtimes f_{p,0}=f_{0,0}$, as it should be.

\subsection{Example 2}\label{SL2(Z)}

Let $N = (\mathbb{Z}/ 12\mathbb{Z})^2$ and $Q=SL_2(\mathbb Z)$ with the usual action of matrices on vectors. We now want to consider an extension that does not split. First, we can deduce that $H^2(SL_2(\mathbb Z), (\mathbb{Z}/ 12\mathbb{Z})^2)=\mathbb{Z}/ 2\mathbb{Z}$ using the well-known exact sequence (compare \cite[Ex. 3, p. 52]{br})
$$\cdots\to H^i(SL_2(\mathbb Z),(\mathbb{Z}/ 12\mathbb{Z})^2)\to H^i(C_4, (\mathbb{Z}/ 12\mathbb{Z})^2)\oplus H^i(C_6, (\mathbb{Z}/ 12\mathbb{Z})^2)$$ $$\to H^i(C_2, (\mathbb{Z}/ 12\mathbb{Z})^2)\to H^{i+1}(SL_2(\mathbb Z), (\mathbb{Z}/ 12\mathbb{Z})^2)\to \cdots .$$

Now we define
$$G:=<a,b,x,y \ | \ a^{12}=1, \ b^{12}=1, \ ab=ba, \ x^4=1, \ y^6=1, \ x^2=ay^3,$$
$$ xax^{-1}=b^{-1}, \ xbx^{-1}=a, \ yay^{-1}=b^{-1}, \ yby^{-1}=ab>.$$
The extension $0\to  (\mathbb{Z}/ 12\mathbb{Z})^2\to G\xrightarrow{p} SL_2(\mathbb Z)\to 0$ now corresponds to the non-zero element of $H^2(SL_2(\mathbb Z), (\mathbb{Z}/ 12\mathbb{Z})^2)$. The map $p$ is defined by $p(x)= \begin{pmatrix} 0 & 1  \\ -1 & 0 \\ \end{pmatrix}$, $p(y)=\begin{pmatrix} 0 & 1  \\ -1 & 1 \\ \end{pmatrix}$ and $p(a)=p(b)=\begin{pmatrix} 1 & 0  \\ 0 &1 \\ \end{pmatrix}$. Here the matrices $\begin{pmatrix} 0 & 1  \\ -1 & 0 \\ \end{pmatrix}$ and $\begin{pmatrix} 0 & 1  \\ -1 & 1 \\ \end{pmatrix}$ were used as generators of $ SL_2(\mathbb Z)$.

Next, we have $End_{SL_2(\mathbb Z)}((\mathbb{Z}/ 12\mathbb{Z})^2) = \mathbb{Z}/ 12\mathbb{Z}$ and $$End^{(\mathbb{Z}/ 12\mathbb{Z})^2,SL_2(\mathbb{Z})}(G)=Z^1(SL_2(\mathbb{Z}),(\mathbb{Z}/ 12\mathbb{Z})^2)=(\mathbb{Z}/ 12\mathbb{Z})^2.$$

Now, let us look at $End^{SL_2(\mathbb{Z})}_{(\mathbb{Z}/ 12\mathbb{Z})^2}(G)=Z^1(G,(\mathbb{Z}/ 12\mathbb{Z})^2)$. We can check that this is another semidirect product of rings, $End^{SL_2(\mathbb{Z})}_{(\mathbb{Z}/ 12\mathbb{Z})^2}(G)=S\rtimes R$, where $S =\{ (m,n)\in (\mathbb{Z}/ 12\mathbb{Z})^2 \ | \ m+n \  even \} $ and $R=2\mathbb{Z}/ 12\mathbb{Z}$. The right action of $R$ on $(\mathbb{Z}/ 12\mathbb{Z})^2$ is zero, while the left action is given by $t\cdot (m,n)=(tm,tn)$ for $(m,n)\in (\mathbb{Z}/ 12\mathbb{Z})^2$ and $t\in R$.

In terms of endomorphisms, the ring structure on $End^{SL_2(\mathbb{Z})}_{(\mathbb{Z}/ 12\mathbb{Z})^2}(G)$ is as follows:
Let $f_{(m,n),t}\in End^{SL_2(\mathbb{Z})}_{(\mathbb{Z}/ 12\mathbb{Z})^2}(G)$ and let $a,b,x,y$ be the generators of $G$. Then $f_{(m,n),t}(a)=a^{t+1}$, $f_{(m,n),t}(b)=b^{t+1}$, $f_{(m,n),t}(x)=a^mb^nx$ and $f_{(m,n),t}(y)=a^{\frac{2m-t}{2}}b^{\frac{m+n-t}{2}}y$.

We have $f_{(k,l),s}\boxplus f_{(m,n),t}=f_{(k+m,l+n),s+t}$ and $f_{(k,l),s}\boxtimes f_{(m,n),t}=f_{(sm,sn),st}$.

\end{document}